\documentclass[11pt,twoside]{amsart}

\usepackage[all]{xy}
\usepackage{amsmath}
\usepackage{amsfonts}
\usepackage{amssymb}
\usepackage{amsthm}
\usepackage{color}
\usepackage{enumerate}
\usepackage{hyperref}
\usepackage{fullpage}
\usepackage{graphicx}
\usepackage{url}

\theoremstyle{definition}
\newtheorem{defn}{Definition}[section]

\theoremstyle{plain}
\newtheorem{thm}[defn]{Theorem}

\newtheorem{lem}[defn]{Lemma}
\newtheorem{prop}[defn]{Proposition}
\newtheorem{cor}[defn]{Corollary}

\newtheorem{conj}[defn]{Conjecture}

\def\C{\ensuremath{\mathbb{C}}}

\def\P{\ensuremath{\mathbb{P}}}

\def\R{\ensuremath{\mathbb{R}}}
\def\Z{\ensuremath{\mathbb{Z}}}

\def\AA{\ensuremath{\mathcal A}}

\def\FF{\ensuremath{\mathcal F}}

\def\HH{\ensuremath{\mathcal H}}
\def\II{\ensuremath{\mathcal I}}

\def\OO{\ensuremath{\mathcal O}}

\def\QQ{\ensuremath{\mathcal Q}}

\def\TT{\ensuremath{\mathcal T}}
\def\UU{\ensuremath{\mathcal U}}
\def\VV{\ensuremath{\mathcal V}}
\def\WW{\ensuremath{\mathcal W}}

\def\ch{\mathop{\mathrm{ch}}\nolimits}

\def\Coh{\mathop{\mathrm{Coh}}\nolimits}

\def\Db{\mathop{\mathrm{D}^{\mathrm{b}}}\nolimits}
\def\deg{\mathop{\mathrm{deg}}\nolimits}

\def\dim{\mathop{\mathrm{dim}}\nolimits}

\def\ext{\mathop{\mathrm{ext}}\nolimits}
\def\Ext{\mathop{\mathrm{Ext}}\nolimits}
\def\lExt{\mathop{\mathcal Ext}\nolimits}

\def\GL{\mathop{\mathrm{GL}}\nolimits}
\def\Gr{\mathop{\mathrm{Gr}}\nolimits}

\def\Hom{\mathop{\mathrm{Hom}}\nolimits}
\def\lHom{\mathop{\mathcal Hom}\nolimits}
\def\RlHom{\mathop{\mathbf{R}\mathcal Hom}\nolimits}
\def\RHom{\mathop{\mathbf{R}\mathrm{Hom}}\nolimits}
\def\id{\mathop{\mathrm{id}}\nolimits}

\def\tilt{\mathop{\mathrm{tilt}}}

\def\td{\mathop{\mathrm{td}}\nolimits}

\def\into{\ensuremath{\hookrightarrow}}
\def\onto{\ensuremath{\twoheadrightarrow}}

\begin{document}

\title{Rank two sheaves with maximal third Chern character in three-dimensional projective space}

\author{Benjamin Schmidt}
\address{The University of Texas at Austin, Department of Mathematics, 2515 Speedway, RLM 8.100, Austin, TX 78712, USA}
\email{schmidt@math.utexas.edu}
\urladdr{https://sites.google.com/site/benjaminschmidtmath/}



\begin{abstract}
We give a complete classification of semistable rank two sheaves on three-dimensional projective space with maximal third Chern character. This implies an explicit description of their moduli spaces. As an open subset they contain rank two reflexive sheaves with maximal number of singularities. These spaces are irreducible, and apart from a single special case, they are also smooth. This extends a result by Okonek and Spindler to all missing cases and gives a new proof of their result. The key technical ingredient is variation of stability in the derived category.
\end{abstract}

\maketitle

\tableofcontents

\section{Introduction}

Moduli spaces of sheaves are well known to be badly behaved. In \cite{Mum62:pathologies} Mumford described a generically non-reduced irreducible component of the Hilbert scheme of curves in $\P^3$ whose general point parametrizes a smooth curve. The fact that well-behaved geometric objects could have such a disastrous moduli space was a shocking result. In \cite{Vak06:murphys_law} this was vastly generalized. Vakil showed that many classes of moduli spaces satisfy Murphy's law in algebraic geometry. This means every possible singularity can occur on them. The moral of these results is that moduli space are problematic, unless there is a good reason to believe otherwise.

In this article, we deal with moduli spaces of rank two sheaves in $\P^3$ whose third Chern character is maximal. They defy the general principle. Except for one case they turn out to be smooth and irreducible. We denote the moduli space of Gieseker-semistable sheaves $E \in \Coh(\P^3)$ with Chern character $\ch(E) = v$ as $M(v)$.

\begin{thm}
\label{thm:main}
Let $E \in \Coh(\P^3)$ be a Gieseker-semistable rank two object with $\ch(E) = (2,c,d,e)$.
\begin{enumerate}
    \item If $c = -1$, then $d \leq -\tfrac{1}{2}$.
    \begin{enumerate}
        \item If $d = -\tfrac{1}{2}$, then $e \leq \tfrac{5}{6}$. Moreover, $M(2,-1,-\tfrac{1}{2},\tfrac{5}{6}) \cong \P^3$.
        \item If $d \leq -\tfrac{3}{2}$, then $e \leq \tfrac{d^2}{2} - d + \tfrac{5}{24}$. Moreover, there is a locally trivial fibration $M(2,-1,d,\tfrac{d^2}{2} - d + \tfrac{5}{24}) \to \P^3$, where the fiber is the Grassmannian $\Gr(2, n)$ for 
        \[
        n = \binom{\frac{5}{2} - d}{2}. 
        \]
    \end{enumerate}
    \item If $c = 0$, then $d \leq 0$.
    \begin{enumerate}
        \item If $d = 0$, then $e \leq 0$. In case of equality, $E \cong \OO^{\oplus 2}$.
        \item If $d = -1$, then $e \leq 0$. Moreover, $M(2,0,-1,0) \cong \P^5$.
        \item If $d = -2$, then $e \leq 2$.
        \item If $d = -3$, then $e \leq 4$. The moduli space $M(2,0,-3,4)$ is the blow up of $\Gr(3,10)$ in a smooth subvariety isomorphic to $\P^3 \times \P^3$.
        \item If $d \leq -4$, then $e \leq \tfrac{d^2}{2} + \tfrac{d}{2} + 1$. Moreover, the moduli space $M(2,0,d,\tfrac{d^2}{2} + \tfrac{d}{2} + 1)$ is a $\P^n$-bundle over $\P^3 \times \P^3$, where $n = d(d-2) - 1$.
    \end{enumerate}
\end{enumerate}
\end{thm}

It turns out that in the case $\ch(E) = (2,0,-2,2)$ there are strictly semistable sheaves preventing the moduli space from being smooth. This case had already been deeply analyzed in \cite{MT94:moduli_sheaves_p3}. A special case of a theorem by Hartshorne in \cite{Har88:stable_reflexive_3} proves these bounds with the extra assumption that $E$ is reflexive. In \cite{OS85:spectrum_torsion_free_sheavesII} Okonek and Spindler proved the same bounds for all semistable sheaves of rank two. Moreover, they described the moduli space if either $c = 0$ and $d \leq -6$ or  $c = -1$ and $d \leq -\tfrac{11}{2}$. The above theorem fills in all the remaining special cases. Our proof is completely independent of these previous results.

Assume that $E$ is reflexive. Then by \cite[Proposition 2.6]{Har80:reflexive_sheaves} having maximal third Chern character can be interpreted as $E$ having the maximal possible number of singularities, i.e., points where it fails to be a vector bundle. Curiously, making $E$ less close to a vector bundle leads to a nice moduli space. Other descriptions of components of moduli spaces of semistable rank two sheaves in $\P^3$ have been found in examples such as \cite{AJT18:components_rank_two_p3, AJTT17:new_moduli, Cha83:stable_rank_two, JMT17:instanton_sheaves, JMT17:infinite_series_rank_two, Man81:rank_two_p3}. So called instanton bundles satisfy additional cohomology vanishings. A long standing conjecture that these bundles represent smooth points in their moduli spaces was settled in \cite{JV14:instanton}. However, all these cases are different, since their closures in the moduli space of all semistable sheaves are well known to contain many singularities.

The situation is similar to the Hilbert scheme of plane curves of degree $d$ in $\P^3$. Interpreted as the moduli space of its ideal sheaves, they also maximize the third Chern character. The rank one and two examples lead us to make the following slightly adventurous conjecture.

\begin{conj}
\label{conj:smooth_moduli}
Let $r \in \Z_{\geq 0}$, $c \in \Z$, $d \in \tfrac{1}{2} \Z$, $e \in \tfrac{1}{6} \Z$ such that $M(r,c,d,e)$ is non-empty, but $M(r,c,d,e') = \emptyset$ for all $e' > e$. If all Gieseker-semistable sheaves with Chern character $(r,c,d,e)$ are Gieseker-stable, then $M(r,c,d,e)$ is smooth and irreducible.
\end{conj}

\subsection{Ingredients}

Our proof is fundamentally different than what was done before in \cite{Har88:stable_reflexive_3} and \cite{OS85:spectrum_torsion_free_sheavesII}. We use the notion of tilt stability (see Section \ref{sec:prelim}) in the derived category due to \cite{Bri08:stability_k3, AB13:k_trivial, BMT14:stability_threefolds}. It generalizes the notion of slope stability by varying the abelian category from coherent sheaves to certain categories of two-term complexes $\Coh^{\beta}(\P^3) \subset \Db(\P^3)$ dependent on a real parameter $\beta$. A new slope function $\nu_{\alpha, \beta}$ depends on another positive real number $\alpha$. Varying $\alpha, \beta$ varies the set of semistable objects. The key is that for $\alpha \gg 0$ and $\beta \ll 0$ all Gieseker-semistable sheaves are $\nu_{\alpha, \beta}$-semistable. In the upper half-plane parametrized by $\alpha > 0$ and $\beta \in \R$, there is a locally-finite wall and chamber structure such that the set of semistable objects is constant within each chamber.

In \cite{Mac14:conjecture_p3} Macr\`i proves an inequality for the Chern characters of tilt-semistable objects (see Theorem \ref{thm:p3_conjecture}). This inequality can be used to show that if $\ch_3(E)$ is larger than or equal to the claimed bound, then $E$ has to be destabilized somewhere in tilt stability. If $\ch_3(E)$ is strictly larger, then we get a contradiction by showing that there is no wall by mostly numerical arguments. In case of equality, we show that there is a unique wall unless $\ch_1(E) = 0$ and $\ch_2(E) = 3$. This unique wall leads to a classification of all semistable sheaves, and the description of the moduli spaces follows. The special case uses techniques close to what was done for some Hilbert schemes of curves in \cite{Sch15:stability_threefolds} and \cite{GHS16:elliptic_quartics}.

\subsection*{Acknowledgments}
I would like to thank Tom Bridgeland, Izzet Coskun, Sean Keel, Marcos Jardim, Emanuele Macr\`i, Ciaran Meachan, and Benjamin Sung  for very useful discussions. This work has been supported by an AMS-Simons Travel Grant.

\subsection*{Notation}

\begin{center}
   \begin{tabular}{ r l }
     $\Db(\P^3)$ & bounded derived category of coherent sheaves on $\P^3$ over $\C$ \\
     $\HH^{i}(E)$ & the $i$-th cohomology group of a complex $E \in \Db(\P^3)$ \\
     $H^i(E)$ & the $i$-th sheaf cohomology group of a complex $E \in \Db(\P^3)$ \\
     $\ch(E)$ & Chern character of an object $E \in \Db(\P^3)$  \\
     $\ch_{\leq l}(E)$ & $(\ch_0(E), \ldots, \ch_l(E))$
   \end{tabular}
\end{center}

\section{Preliminaries}
\label{sec:prelim}

In this section, we will recall several notions of stability in the bounded derived category of coherent sheaves and its basic properties. By abuse of notation, we write $\ch_i(E)$ for $E \in \Db(\P^3)$ but mean its intersection with $3-i$ powers of the hyperplane class.

\subsection{Stability for sheaves}

\begin{defn}
\begin{enumerate}
\item The classical \emph{slope} for a coherent sheaf $E \in \Coh(\P^3)$ is defined as
\[
\mu(E) := \frac{\ch_1(E)}{\ch_0(E)},
\]
where division by zero is interpreted as $+\infty$.
\item A coherent sheaf $E$ is called \emph{slope-(semi)stable} if for any non-trivial proper subsheaf $F \into E$ the inequality $\mu(F) < (\leq) \mu(E/F)$ holds.
\end{enumerate}
\end{defn}

In many cases, this notion is not quite refined enough.

\begin{defn}
Let $f, g \in \R[m]$ be two polynomials.
\begin{enumerate}
\item If $\deg(f) < \deg(g)$, then $f > g$. Vice versa, $\deg(g) < \deg(f)$ implies $g > f$.
\item Assume $d = \deg(f) = \deg(g)$, and let $a$, $b$ be the coefficients of $m^d$ in $f$, $g$. Then we define $f < (\leq) g$ if $\tfrac{f(m)}{a} < (\leq) \tfrac{g(m)}{b}$ for all $m \gg 0$. 
\end{enumerate}
\end{defn}

\begin{defn}
\begin{enumerate}
\item For any $E \in \Coh(\P^3)$ we can define numbers $\alpha_i(E)$ for $i \in \{0, 1, 2, 3\}$ via the Hilbert polynomial
\[
P(E, m) := \chi(E(m)) = \alpha_3(E) m^3 + \alpha_2(E) m^2 + \alpha_1(E) m + \alpha_0(E).
\]
Moreover, we set
\[
P_2(E, m) := \alpha_3(E) m^2 + \alpha_2(E) m + \alpha_1(E).
\]
\item A sheaf $E \in \Coh(\P^3)$ is called \emph{Gieseker-(semi)stable} if for any non-trivial proper subsheaf $F \into E$ the inequality $P(F, m) < (\leq) P(E/Fdocument, m)$ holds.
\item A sheaf $E \in \Coh(\P^3)$ is called \emph{$2$-Gieseker-(semi)stable} if for any non-trivial proper subsheaf $F \into E$ the inequality $P_2(F, m) < (\leq) P_2(E/F, m)$ holds.
\end{enumerate}
\end{defn}

Gieseker stability was introduced by Gieseker for torsion-free sheaves and later generalized to torsion sheaves by Simpson. The notion of $2$-Gieseker stability is less known, but it is in fact the precise notion that we need to connect it to tilt stability as described in the next subsection. Finally, there are the following relations between these notions.

\centerline{
\xymatrix{
\text{slope-stable} \ar@{=>}[r] & \text{$2$-Gieseker-stable} \ar@{=>}[r] & \text{Gieseker-stable} \ar@{=>}[d] \\
\text{slope-semistable} & \text{$2$-Gieseker-semistable} \ar@{=>}[l] &\text{Gieseker-semistable} \ar@{=>}[l]
}}

\subsection{Tilt Stability}

By using the derived category it is possible to obtain a more flexible form of stability. The key idea due to Bridgeland is to change the category of coherent sheaves for another heart of a bounded t-structure inside the bounded derived category.

The following notion of \emph{tilt stability} was introduced by Bridgeland for K3 surfaces \cite{Bri08:stability_k3}, later generalized to all surfaces by Arcara-Bertram \cite{AB13:k_trivial}, and finally extended to higher dimensions by Bayer-Macr\`i-Toda in \cite{BMT14:stability_threefolds}. 

Let $\beta$ be an arbitrary real number. Then the twisted Chern character $\ch^{\beta}$ is defined to be $e^{-\beta H} \cdot \ch$, where $H$ is the hyperplane class. Note that for $\beta \in \Z$ one has $\ch^{\beta}(E) = \ch(E(-\beta))$ for any $E \in \Db(\P^3)$. Explicitly:
\begin{align*}
\ch^{\beta}_0 &=  \ch_0, \ \ch^{\beta}_1 = \ch_1 - \beta \ch_0, \ \ch^{\beta}_2 = \ch_2 - \beta \ch_1 + \frac{\beta^2}{2} \ch_0,\\
\ch^{\beta}_3 &= \ch_3 - \beta \ch_2 + \frac{\beta^2}{2} \ch_1 - \frac{\beta^3}{6} \ch_0.
\end{align*}

The process of tilting is used to construct a new heart of a bounded t-structure. For more information on the general theory of tilting we refer to \cite{BvdB03:functors, HRS96:tilting}.

\begin{defn}
\begin{enumerate}
\item A torsion pair is defined by
\begin{align*}
\TT_{\beta} &:= \{E \in \Coh(\P^3) : \text{any quotient $E \onto G$ satisfies $\mu(G) > \beta$} \}, \\
\FF_{\beta} &:=  \{E \in \Coh(\P^3) : \text{any non-trivial subsheaf $F \subset E$ satisfies $\mu(F) \leq \beta$} \}.
\end{align*}
We define $\Coh^{\beta}(\P^3)$ as the extension closure $\langle \FF_{\beta}[1],\TT_{\beta} \rangle$.
\item Let $\alpha > 0$ be a positive real number. The \emph{tilt-slope} is defined as
\[
\nu_{\alpha, \beta} := \frac{\ch^{\beta}_2 - \frac{\alpha^2}{2} \ch^{\beta}_0}{\ch^{\beta}_1}.
\]
\item Similarly as before, an object $E \in \Coh^{\beta}(\P^3)$ is called \emph{tilt-(semi)stable} (or \emph{$\nu_{\alpha,\beta}$-(semi)stable}) if for any non-trivial proper subobject $F \into E$ the inequality $\nu_{\alpha, \beta}(F) < (\leq) \nu_{\alpha, \beta}(E/F)$ holds.
\end{enumerate}
\end{defn}

Note that $\Coh^{\beta}(\P^3)$ consists of some two term complex. More precisely, it contains exactly those complexes $E \in \Db(\P^3)$ such that $\HH^0(E) \in \TT_{\beta}$, $\HH^{-1}(E) \in \FF_{\beta}$, and $\HH^i(E) = 0$ whenever $i \neq -1, 0$. The following proposition was proved for K3 surfaces in \cite[Proposition 14.2]{Bri08:stability_k3}, but the proof holds without trouble in our case.

\begin{prop}
\label{prop:large_volume_limit}
If $E \in \Coh^{\beta}(\P^3)$ for $\beta < \mu(E)$ is $\nu_{\alpha, \beta}$-(semi)stable for $\alpha \gg 0$, then it is a $2$-Gieseker-(semi)stable sheaf. Vice versa, if $E$ is a $2$-Gieseker-(semi)stable sheaf, then it is $\nu_{\alpha, \beta}$-semistable for any $\alpha \gg 0$ and $\beta < \mu(E)$.
\end{prop}

Note that in case $\mu(E) = \infty$, this proposition holds for arbitrary $\beta$. The Chern characters of semistable objects satisfy certain inequalities. The first one is well known as the Bogomolov inequality. It was first proved for slope-semistable sheaves on surfaces (\cite{Rei78:bogomolov, Bog78:inequality, Gie79:bogomolov}).

\begin{thm}[{Bogomolov inequality for tilt stability, \cite[Corollary 7.3.2]{BMT14:stability_threefolds}}]
Any $\nu_{\alpha, \beta}$-semistable object $E \in \Coh^{\beta}(\P^3)$ satisfies
\[
\Delta(E) := \ch_1(E)^2 - 2 \ch_0(E) \ch_2(E) \geq 0.
\]
\end{thm}

The following inequality involving the third Chern character is part of a more general conjecture in \cite{BMT14:stability_threefolds} and was brought into the following form in \cite{BMS16:abelian_threefolds}. The case of $\P^3$ was proved in \cite{Mac14:conjecture_p3}

\begin{thm}
\label{thm:p3_conjecture}
For any $\nu_{\alpha,\beta}$-semistable object $E \in \Coh^{\beta}(\P^3)$ the inequality
\[
Q_{\alpha, \beta}(E) := \alpha^2 \Delta(E) + 4\ch_2^{\beta}(E)^2 - 6\ch_1^{\beta}(E) \ch_3^{\beta}(E) \geq 0
\]
holds.
\end{thm}

We will need to understand interactions between the derived dual and tilt stability.

\begin{prop}[{\cite[Proposition 5.1.3]{BMT14:stability_threefolds}}]
\label{prop:tilt_derived_dual}
Assume $E \in \Coh^{\beta}(\P^3)$ is $\nu_{\alpha, \beta}$-semistable with $\nu_{\alpha, \beta}(E) \neq \infty$. Then there is a $\nu_{\alpha, -\beta}$-semistable object $\tilde{E} \in \Coh^{-\beta}(\P^3)$ and a sheaf $T$ supported in dimension zero together with a distinguished triangle
\[
\tilde{E} \to \RlHom(E, \OO)[1] \to T[-1] \to \tilde{E}[1].
\]
\end{prop}

\subsection{Walls}

Let $\Lambda = \Z \oplus \Z \oplus \tfrac{1}{2} \Z$. Then $H \cdot \ch_{\leq 2}$ maps to $\Lambda$. Varying $(\alpha, \beta)$ changes the set of semistable objects. A \textit{numerical wall} in tilt stability with respect to a class $v \in \Lambda$ is a non-trivial proper subset $W$ of the upper half plane given by an equation of the form $\nu_{\alpha, \beta}(v) = \nu_{\alpha, \beta}(w)$ for another class $w \in \Lambda$. We will usually write $W = W(v, w)$.

A subset $S$ of a numerical wall $W$ is called an \textit{actual wall} if the set of semistable objects with class $v$ changes at $S$. The structure of walls in tilt stability is rather simple. Part (i) - (v) is usually called Bertram's Nested Wall Theorem and appeared in \cite{Mac14:nested_wall_theorem}, while part (vi) and (vii) can be found in \cite[Appendix A]{BMS16:abelian_threefolds}.

\begin{thm}[Structure Theorem for Walls in Tilt Stability]
\label{thm:Bertram}
Let $v \in \Lambda$ be a fixed class. All numerical walls in the following statements are with respect to $v$.
\begin{enumerate}
  \item Numerical walls in tilt stability are either semicircles with center on the $\beta$-axis or rays parallel to the $\alpha$-axis. If $v_0 \neq 0$, there is exactly one numerical vertical wall given by $\beta = v_1/v_0$. If $v_0 = 0$, there is no actual vertical wall.
   \item The curve $\nu_{\alpha, \beta}(v) = 0$ is given by a hyperbola, which may be degenerate if $v_0 = 0$. Moreover, this hyperbola intersects all semicircular walls at their top point.
  \item If two numerical walls given by classes $w,u \in \Lambda$ intersect, then $v$, $w$ and $u$ are linearly dependent. In particular, the two walls are completely identical.
  \item If a numerical wall has a single point at which it is an actual wall, then all of it is an actual wall.
  \item If $v_0 \neq 0$, then there is a largest semicircular wall on both sides of the unique numerical vertical wall. If $v_0 = 0$, then there is a unique largest semicircular wall.
  \item If there is an actual wall numerically defined by an exact sequence of tilt-semistable objects $0 \to F \to E \to G \to 0$ such that $\ch_{\leq 2}(E) = v$, then 
  \[\Delta(F) + \Delta(G) \leq \Delta(E).\]
  Moreover, equality holds if and only if $\ch_{\leq 2}(G) = 0$.
  \item If $\Delta(E) = 0$, then $E$ can only be destabilized at the unique numerical vertical wall. In particular, line bundles, respectively their shifts by one, are tilt-semistable everywhere.
\end{enumerate}
\end{thm}

If $W = W(v,w)$ is a semicircular wall in tilt stability for two numerical classes $v, w \in \Lambda$, then we denote its \emph{radius} by $\rho_W = \rho(v,w)$ and its \emph{center} on the $\beta$-axis by $s_W = s(v,w)$. The structure of the locus $Q_{\alpha, \beta}(E) = 0$ fits right into into the semicircle wall picture. Indeed, a straightforward computation shows the following.

\begin{lem}
Let $E \in \Db(\P^3)$. The equation $Q_{\alpha, \beta}(E) = 0$ is equivalent to
\[
\nu_{\alpha, \beta}(E) = \nu_{\alpha, \beta}(\ch_1(E), 2 \ch_2(E), 3\ch_3(E)).
\]
In particular, $Q_{\alpha, \beta}(E) = 0$ describes a numerical wall in tilt stability.
\end{lem}

We denote the numerical wall $Q_{\alpha, \beta}(E) = 0$ by $W_Q = W_Q(E)$, its radius by $\rho_Q = \rho_Q(E)$, and its center by $s_Q = s_Q(E)$. The following lemma is a highly convenient tool to control the rank of destabilizing subobjects. A very close version appeared first in \cite[Proposition 8.3]{CH16:ample_cone_plane}.

\begin{lem}[{\cite[Lemma 2.4]{MS18:space_curves}}]
\label{lem:higherRankBound}
Assume that a tilt-semistable object $E$ is destabilized by either a subobject $F \into E$ or a quotient $E \onto F$ in $\Coh^{\beta}(\P^3)$ inducing a non-empty semicircular wall $W$. Assume further that $\ch_0(F) > \ch_0(E) \geq 0$. Then the inequality
\[
\rho_W^2 \leq \frac{\Delta(E)}{4 \ch_0(F) (\ch_0(F) - \ch_0(E))}
\]
holds.
\end{lem}

If we understand the radius $\rho_Q(E)$, this lemma will provide a key tool to control the rank of destabilizing subobjects.

\subsection{Bridgeland stability}

Tilt stability has well-behaved computational properties. However, for dimension greater than or equal to three it does not have well-behaved moduli spaces. This is similar to the issues of slope stability on surfaces, where Gieseker stability turns out to be the better notion. We need to recall the constructions of Bridgeland stability on $\P^3$ due to \cite{BMT14:stability_threefolds,Mac14:conjecture_p3}. The idea is to perform another tilt as previously. Let
\begin{align*}
\TT'_{\alpha, \beta} &:= \{E \in \Coh^{\beta}(\P^3) : \text{any quotient $E \onto G$ satisfies $\nu_{\alpha, \beta}(G) > 0$} \}, \\
\FF'_{\alpha, \beta} &:=  \{E \in \Coh^{\beta}(\P^3) : \text{any non-trivial subobject $F \into E$ satisfies $\nu_{\alpha, \beta}(F) \leq 0$} \}
\end{align*}
and set $\AA^{\alpha, \beta}(\P^3) := \langle \FF'_{\alpha, \beta}[1], \TT'_{\alpha, \beta} \rangle $. For any $s > 0$ they define
\begin{align*}
Z_{\alpha,\beta,s} &:= -\ch^{\beta}_3 + (s+\tfrac{1}{6})\alpha^2 \ch^{\beta}_1 + i (\ch^{\beta}_2 - \frac{\alpha^2}{2} \ch^{\beta}_0), \\
\lambda_{\alpha,\beta,s} &:= -\frac{\Re(Z_{\alpha,\beta,s})}{\Im(Z_{\alpha,\beta,s})}.
\end{align*}

\begin{defn}
An object $E \in \AA^{\alpha, \beta}(\P^3)$ is called $\lambda_{\alpha, \beta, s}$-(semi)stable if for any non-trivial subobject $F \into E$, we have $\lambda_{\alpha, \beta, s}(F) < (\leq) \lambda_{\alpha, \beta, s}(E)$.
\end{defn}

\subsection{Moduli Spaces}

For any $v \in K_0(\P^3)$, $\alpha > 0$, $\beta \in \R$, $s > 0$ we make the following definitions.

\begin{defn}
\begin{enumerate}
    \item The moduli space of slope-semistable sheaves with Chern character $v$ is denoted by $M(v)$. 
    \item The moduli space of $\nu_{\alpha, \beta}$-semistable objects with Chern character $v$ is denoted $M^{\tilt}_{\alpha, \beta}(v)$.
    \item The moduli space of $\lambda_{\alpha, \beta, s}$-semistable objects with Chern character $v$ is denoted $M^{B}_{\alpha, \beta, s}(v)$.
\end{enumerate}
\end{defn}

For some Bridgeland stability conditions it is possible to exchange the heart $\AA^{\alpha, \beta}(\P^3)$ by a finite length category.

\begin{thm}[\cite{Mac14:conjecture_p3}]
\label{thm:macri_p3}
Let $\alpha < 1/3$, $\beta \in (-5/3, -1]$, and $0 < s \ll 1$. For any $\gamma \in \R$, we can define a torsion pair
\begin{align*}
\TT''_{\gamma} &:= \{E \in \AA^{\alpha, \beta}(\P^3) : \text{any quotient $E \onto G$ satisfies $\lambda_{\alpha, \beta, s}(G) > \gamma$} \}, \\
\FF''_{\gamma} &:=  \{E \in \AA^{\alpha, \beta}(\P^3) : \text{any non-trivial subobject $F \into E$ satisfies $\lambda_{\alpha, \beta, s}(F) \leq \gamma$} \}.
\end{align*}
There is a choice of $\gamma \in \R$ such that
\[
\langle \TT''_{\gamma}, \FF''_{\gamma}[1] \rangle = \langle \mathcal{O}(-2)[3], T(-3)[2], \mathcal{O}(-1)[1],
\mathcal{O} \rangle.
\]
In particular, moduli spaces of Bridgeland-stable objects for these special choices of stability conditions are the same as moduli spaces of representation of finite-dimensional algebras as defined in \cite{Kin94:moduli_quiver_reps}. This means they have projective moduli spaces.
\end{thm}

Through most of this article, we will only study tilt stability. The following statement makes a connection to Bridgeland stability.

\begin{thm}[{\cite[Theorem 6.1(3)]{Sch15:stability_threefolds}}]
\label{thm:Bridgeland_stability_hyperbola}
Let $v \in K_0(\P^3)$, $\alpha_0 > 0$, $\beta_0 \in \R$, and $s > 0$ such that $\nu_{\alpha_0, \beta_0}(v) = 0$, $H^2 \cdot v^{\beta_0}_1 > 0$, and $\Delta(v) \geq 0$. Assume that all $\nu_{\alpha_0, \beta_0}$-semistable objects of class $v$ are $\nu_{\alpha_0, \beta_0}$-stable. Then there is a neighborhood $U$ of $(\alpha_0, \beta_0)$ such that 
\[
M^{B}_{\alpha, \beta, s}(v) = M^{\tilt}_{\alpha, \beta}(v)
\]
for all $(\alpha, \beta) \in U$ with $\nu_{\alpha, \beta}(v) > 0$. Moreover, in this case all objects parametrized in $M^{B}_{\alpha, \beta, s}(v)$ are $\lambda_{\alpha, \beta, s}$-stable.
\end{thm}

The following result by Piyaratne and Toda is a major step towards the construction of well-behaved moduli spaces. It applies in particular to the case of $\P^3$, since the conjectural BMT-inequality is known.

\begin{thm}[\cite{PT15:bridgeland_moduli_properties}]
\label{thm:moduli_bridgeland_stable_objects}
Let $X$ be a smooth projective threefold such that the conjectural construction of Bridgeland stability from \cite{BMT14:stability_threefolds} works. Then any moduli space of semistable objects for such a Bridgeland stability condition is a universally closed algebraic stack of finite type over $\C$. If there are no strictly semistable objects, the moduli space becomes a proper algebraic space of finite type over $\C$.
\end{thm}

\subsection{Some known bounds}

We recall further known results about Chern character bounds for tilt stability in $\P^3$.

\begin{lem}[{\cite[Lemma 5.4]{Sch15:stability_threefolds}}]
\label{lem:simple_cases}
\begin{enumerate}
    \item Let $E \in \Coh^{\beta}(\P^3)$ be tilt-semistable with $\ch(E) = (1,0,-1,1)$. Then $E \cong \II_L$ for a line $L \subset \P^3$.
    \item Let $E \in \Coh^{\beta}(\P^3)$ be tilt-semistable with $\ch(E) = (0,1,d,e)$, then $E \cong \II_{Z/V}(d + 1/2)$ where $Z$ is a dimension zero subscheme of length $\tfrac{1}{24} + \tfrac{d^2}{2} - e$.
\end{enumerate}
\end{lem}

\begin{prop}
\label{prop:rank_one_old}
Let $E \in \Coh^{\beta}(\P^3)$ be a $\nu_{\alpha, \beta}$-semistable object for some $(\alpha, \beta)$ with either $\ch(E) = (1,0,-d,e)$ or $\ch(E) = (-1,0,d,e)$. Then
\[
e \leq \frac{d(d+1)}{2}.
\]
In case of equality and $\ch_0(E) = 1$, the object $E$ is the ideal sheaf of a plane curve.
\end{prop}

\begin{proof}
The bounds where shown in \cite[Proposition 3.2]{MS18:space_curves}. If $\ch_0(E) = 1$, then the proof also shows that in case of equality, $E$ is destabilized by a morphism $\OO(-1) \into E$ unless $\ch_0(E) = 1$ and $d = 1$. The special case is solved directly by Lemma \ref{lem:simple_cases}. Let $G = E/\OO(-1)$ be the quotient. Then 
\[
\ch(G) = \left(0, 1, -d - \frac{1}{2}, \frac{d(d+1)}{2} + \frac{1}{6} \right).
\]
By Lemma \ref{lem:simple_cases}, we get $G \cong \OO_V(-d)$, and the claim follows.
\end{proof}

Line bundles can be easily identified by their Chern characters as follows. This was shown in \cite[Proposition 4.1, 4.5]{Sch15:stability_threefolds}.

\begin{prop}
\label{prop:line_bundles_uniquely_stable}
Let $E$ be a tilt-semistable or Bridgeland-semistable object. Assume that there are integers $n,m$ with $m > 0$ such that either
\begin{enumerate}
\item $v = m \ch(\OO(n))$, or
\item $v = -m\ch_{\leq 2}(\OO(n))$.
\end{enumerate}
Then $E \cong \OO(n)^{\oplus m}$ or a shift of it. Moreover, in the case $m=1$ the line bundle $\OO(n)$ is stable. 
\end{prop}

The following statement from \cite[Theorem 3.4]{MS18:space_curves} had a further error term which we will not need. For the convenience of the reader, we will give a proof, since it is shorter and substantially easier without the error term.

\begin{thm}
\label{thm:rank_zero_bound}
Let $E \in \Coh^{\beta}(\P^3)$ be a $\nu_{\alpha, \beta}$-semistable object with $\ch(E) = (0,c,d,e)$, where $c > 0$. Then 
\[
e \leq \frac{c^3}{24} + \frac{d^2}{2c}.
\]
\end{thm}

\begin{proof}
Note that if a subobject of rank zero destabilizes $E$, then it does so independently of $(\alpha, \beta)$. Therefore, any wall must be induced by a subobject or quotient of positive rank. By Lemma \ref{lem:higherRankBound} this means that any wall $W$ must have radius satisfying
\[
\rho_W \leq \frac{c}{2}.
\]
Therefore, $E$ has to be tilt-semistable for some $(\alpha, \beta)$ inside or on the semicircular wall with radius $\tfrac{c}{2}$. But for such $(\alpha, \beta)$ the inequality $Q_{\alpha, \beta}(E) \geq 0$ implies the statement.
\end{proof}


\section{Classifying rank two reflexive sheaves with maximal third Chern character}

Before stating the theorem we require further notation. Let $\alpha > 0$, $\beta \in \R$. Any $\nu_{\alpha, \beta}$-semistable object $E \in \Coh^{\beta}(\P^3)$ has a Jordan-H\"older filtration
\[
0 = E_0 \to E_1 \to \ldots \to E_n = E,
\]
where all the factors $F_i = E_i/E_{i+1}$ are $\nu_{\alpha, \beta}$-stable and have slope $\nu_{\alpha, \beta}(F_i) = \nu_{\alpha, \beta}(E)$. In Bridgeland stability Jordan-H\"older factors are unique up to order. The same is not true in tilt stability. This is a serious issue that we will have to deal with. We say that that $E$ satisfies the \emph{JH-property with respect to $(\alpha, \beta)$} if the Jordan-H\"older factors of $E$ are unique up to order.

\begin{thm}
\label{thm:sequences_maximal_ch3}
Let $E \in \Coh^{\beta}(\P^3)$ be a tilt-semistable rank two object with $\ch(E) = (2,c,d,e)$.
\begin{enumerate}
    \item If $c = -1$, then $d \leq -\tfrac{1}{2}$.
    \begin{enumerate}
        \item If $d = -\tfrac{1}{2}$, then $e \leq \tfrac{5}{6}$. In case of equality, $E$ is destabilized by a short exact sequence
        \[
        0 \to \OO(-1)^{\oplus 3} \to E \to \OO(-2)[1] \to 0.
        \]
        Moreover, $E$ satisfies the JH-property along the wall.
        \item If $d < -\tfrac{1}{2}$, then
        \[
        e \leq \frac{d^2}{2} - d + \frac{5}{24}.
        \]
        In case of equality, $E$ is destabilized by a short exact sequence
        \[
        0 \to \OO(-1)^{\oplus 2} \to E \to \OO_V\left(d - \frac{1}{2}\right) \to 0,
        \]
        where $V \subset \P^3$ is a plane. Moreover, $E$ satisfies the JH-property the wall.
    \end{enumerate}
    \item If $c = 0$, then $d \leq 0$.
    \begin{enumerate}
        \item If $d = 0$, then $e \leq 0$. In case of equality, $E \cong \OO^{\oplus 2}$.
        \item If $d = -1$, then $e \leq 0$. In case of equality, there is a short exact sequence
        \[
        0 \to T(-3) \to \OO(-1)^{\oplus 5} \to E \to 0,
        \]
        where $T$ is the tangent bundle of $\P^3$.
        \item If $d = -2$, then $e \leq 2$. In case of equality, $E$ is destabilized by a short exact sequence
        \[
        0 \to \OO(-1)^{\oplus 4} \to E \to \OO(-2)^{\oplus 2}[1] \to 0.
        \]
        \item If $d = -3$, then $e \leq 4$. In case of equality, $E$ is destabilized by one of the short exact sequences
        \[
        0 \to \OO(-1)^{\oplus 3} \to E \to \OO(-3)[1] \to 0,
        \]
        \[
        0 \to F \to E \to \OO_V(-2) \to 0,
        \]
        \[
        0 \to \OO_V(-2) \to E \to F \to 0,
        \]
        where $V \subset \P^3$ is a plane and $F \in M(2,-1,-\tfrac{1}{2},\tfrac{5}{6})$. Moreover, any semistable $E$ satisfies $E$ satisfies the JH-property along its destabilizing wall.
        \item If $d \leq -4$, then
        \[
        e \leq \frac{d^2}{2} + \frac{d}{2} + 1.
        \]
        In case of equality, $E$ is destabilized via an exact sequence
        \[
        0 \to F \to E \to \OO_V(d + 1) \to 0,
        \]
        where $V \subset \P^3$ is a plane and $F \in M(2,-1,-\tfrac{1}{2},\tfrac{5}{6})$. Moreover, $E$ satisfies the JH-property along the wall.
    \end{enumerate}
\end{enumerate}
\end{thm}

As an immediate Corollary, we can determine the maximal third Chern character for rank $-2$ objects as well.

\begin{cor}
\label{cor:rank_minustwo_bounds_p3}
Let $E \in \Coh^{\beta}(\P^3)$ be a tilt-semistable object with $\ch(E) = (-2,c,d,e)$.
\begin{enumerate}
\item If $c = -1$, then $d \geq \tfrac{1}{2}$ and
\[
e \leq \frac{d^2}{2} + d + \frac{5}{24}.
\]
\item If $c = 0$, then $d \geq 0$.
\begin{enumerate}
\item If $d = 0$ or $d = 1$, then $e \leq 0$.
\item If $d \geq 2$, then
\[
e \leq \frac{d^2}{2} - \frac{d}{2} + 1.
\]
\end{enumerate}
\end{enumerate}
\end{cor}

\begin{proof}
By Proposition \ref{prop:tilt_derived_dual}, there is a sheaf $T$ supported in dimension $0$ and an object $\tilde{E}$ that is tilt-semistable together with a distinguished triangle
\[
T \to \tilde{E} \to \RHom(E, \OO)[1] \to T[1].
\]
We have $\ch(\tilde{E}) = (2, c, -d, e + \ch_3(T))$, and in particular, $e \leq \ch_3(\tilde{E})$. Applying the bounds in Theorem \ref{thm:sequences_maximal_ch3} finishes the proof.
\end{proof}

The strategy to prove Theorem \ref{thm:sequences_maximal_ch3} is induction on the discriminant. We start with a series of special cases that will either serve as base cases or require special arguments. The bounds in some of these special cases, where already dealt with in \cite{SS18:rank_two_surface_bounds}. Nevertheless, for the convenience of the reader, we include full proofs.

\begin{lem}
\label{lem:200_p3}
Let $E \in \Coh^{\beta}(\P^3)$ be a tilt-semistable rank two object with $\ch(E) = (2,0,0,e)$. Then $e \leq 0$, and if $e = 0$, then $E \cong \OO^{\oplus 2}$.
\end{lem}

\begin{proof}
The fact that $E \in \Coh^{\beta}(\P^3)$ implies $\beta < 0$. The inequality $Q_{\alpha, \beta}(E) \geq 0$ is equivalent to $e \leq 0$ for any $\alpha > 0$. For the fact that $e = 0$ implies $E \cong \OO^{\oplus 2}$ see Proposition \ref{prop:line_bundles_uniquely_stable}.
\end{proof}

\begin{lem}
\label{lem:2-1-1/2_p3}
Let $E \in \Coh^{\beta}(\P^3)$ be a tilt-semistable rank two object with $\ch(E) = (2,-1,-\tfrac{1}{2},e)$. Then $e \leq \tfrac{5}{6}$, and if $e = \tfrac{5}{6}$, then $E$ is destabilized in tilt stability by an exact sequence
\[
0 \to \OO(-1)^{\oplus 3} \to E \to \OO(-2)[1] \to 0.
\]
Moreover, $E$ satisfies the JH-property along the wall.
\end{lem}

\begin{proof}
The point $\alpha = 0$, $\beta = -1$ lies on the numerical wall with center $-\tfrac{3}{2}$ and radius $\tfrac{1}{2}$. A straightforward computation shows that for any point inside this numerical wall $Q_{\alpha, \beta}(E) \geq 0$ implies $e < \tfrac{5}{6}$. Therefore, we only have to deal with objects stable at or outside this wall.

We have $\ch^{-1}(E) = (2, 1, -\tfrac{1}{2}, e -\tfrac{2}{3})$. Since $\ch^{-1}_1(E)$ is the minimal positive value, and $\beta = -1$ is not the vertical wall, $E$ must be semistable for all $\alpha > 0$ when $\beta = -1$. Finally, $Q_{0, -1} \geq 0$ implies $e \leq \tfrac{5}{6}$. The remaining statement is a special case of \cite[Theorem 5.1]{Sch15:stability_threefolds}.
\end{proof}

\begin{lem}
\label{lem:20-1_p3}
Let $E \in \Coh^{\beta}(\P^3)$ be a tilt-semistable rank two object with $\ch(E) = (2,0,-1,e)$. Then $e \leq 0$. If $e = 0$, then $E$ fits into an exact sequence of the form
\[
0 \to T(-3) \to \OO(-1)^{\oplus 5} \to E \to 0.
\]
\end{lem}

\begin{proof}
The fact that $E \in \Coh^{\beta}(\P^3)$ implies $\beta < 0$. The equation $\nu_{0, \beta}(E) = 0$ holds if and only if $\beta = \pm 1$. If $E$ is destabilized at a semicircular wall, it must intersect the vertical line $\beta = -1$. We have $\ch^{-1}(E) = (2,2,0,e-\tfrac{2}{3})$. Assume we have such a wall induced by $0 \to F \to E \to G \to 0$ that contains a point $(\alpha, -1)$. Since the wall itself is not vertical, $\ch_1^{-1}(F)$ has to be an integer strictly in between $0$ and $2$, i.e., $\ch_1^{-1}(F) = 1$. We know $\ch^{-1}_{\leq 2}(F) = (r, 1, x)$, where $r \in \Z$ and $x \in \tfrac{1}{2} + \Z$. Then
\[
-\frac{\alpha^2}{2} = \nu_{\alpha, -1}(E) = \nu_{\alpha, -1}(F) = x - \frac{\alpha^2}{2}r.
\]
This simplifies to $(r - 1)\alpha^2 = 2x$. If $r \geq 2$, then $x > 0$ and $\Delta(F) \geq 0$ implies $x \leq \tfrac{1}{2r} \leq \frac{1}{4}$. There is no possible value for $x$ with these properties. If $r = 1$, then $x = 0$ which is not a valid value for $x$. If $r \leq 0$, then the quotient $E/F$ has positive rank. The same argument with $E/F$ instead of $F$ works.

Overall, there is no wall to the left of the unique vertical wall for $E$. If $H^0(E) \neq 0$, we get a non-trivial morphism $\OO \to E$ in contradiction to stability of $E$. If $H^2(E) \neq 0$, we can use Serre duality to get a non-trivial morphism $E \to \OO(-4)[1]$. However, such a morphism would induce a semicircular wall, and therefore, $H^2(E) = 0$.

The Todd class of $\P^3$ is given by
\[
\td(T_{\P^3}) = \left(1,2,\frac{11}{6},1\right).
\]
We get
\[
e = \chi(E) = - h^1(E) - h^3(E) \leq 0.
\]
To prove the second statement about the exact sequence, note that $\nu_{7/24, -25/24}(E) = 0$. By Theorem \ref{thm:p3_conjecture} and Theorem \ref{thm:Bridgeland_stability_hyperbola} the object $E$ or $E[1]$ is in the finite length category
\[
\langle \OO(-2)[3], T(-3)[2], \OO(-1)[1], \OO \rangle.
\]
The Chern character of $E$ directly implies that it has to be an extension between $T(-3)[1]$ and five copies of $\OO(-1)$.
\end{proof}

\begin{lem}
\label{lem:20-2_p3}
Let $E \in \Coh^{\beta}(\P^3)$ be a tilt-semistable rank two object with $\ch(E) = (2,0,-2,e)$. Then $e \leq 2$. If $e = 2$, then $E$ is destabilized in tilt stability by an exact sequence
\[
0 \to \OO(-1)^{\oplus 4} \to E \to \OO(-2)^{\oplus 2}[1] \to 0.
\]
\end{lem}

\begin{proof}
Assume $e \geq 2$. We have $\ch^{-1}(E) = (2,2,-1, e - \tfrac{5}{3})$. We will show that there is no wall for $\beta = -1$. If there is any destabilizing subobject $F \subset E$ for some $\alpha > 0$ and $\beta = -1$, then $\ch^{-1}_1(F) = 1$. Let $\ch_{\leq 2}^{-1}(F) = (r,1,d)$. We may assume $r \geq 1$. If $r \leq 0$, then we can simply replace $F$ by $F/E$ in the following argument.

If $r = 1$, we are dealing with the vertical wall, but that is located at $\beta = 0$. Thus, $r \geq 2$. A straightforward computation shows that the wall occurs for
\[
\alpha^2 = \frac{2d + 1}{r-1}.
\]
This means $d > -\tfrac{1}{2}$. Furthermore, $\Delta(F) \geq 0$ implies
\[
d \leq \frac{1}{2r} < \frac{1}{2}.
\]
This is a contradiction to the fact that $d \in \tfrac{1}{2} + \Z$. The remaining statement is a special case of \cite[Theorem 5.1]{Sch15:stability_threefolds}.
\end{proof}

\begin{lem}
\label{lem:20-3_p3}
Let $E \in \Coh^{\beta}(\P^3)$ be a tilt-semistable rank two object with $\ch(E) = (2,0,-3,e)$. Then $e \leq 4$. If $e = 4$, then $E$ is is destabilized in tilt stability by one of the following of the three following sequences
\[
0 \to \OO(-1)^{\oplus 3} \to E \to \OO(-3)[1] \to 0,
\]
\[
0 \to F \to E \to O_V(-2) \to 0,document
\]
\[
0 \to \OO_V(-2) \to E \to F \to 0,
\]
where $V \subset \P^3$ is a plane and $F \in M(2,-1,-\tfrac{1}{2},\tfrac{5}{6})$. Moreover, any semistable $E$ satisfies the JH-property along these walls.
\end{lem}

\begin{proof}
Assume $e \geq 2$. A straightforward computation shows that $Q_{0,-1}(E) \geq 0$ is equivalent to $e \leq 4$. Therefore, we only need to check this inequality for objects that destabilizes along the vertical ray $\beta = -1$. We have $\ch^{-1}(E) = (2,2,-2, e - \tfrac{8}{3})$. If there is any destabilizing subobject $F \subset E$ for some $\alpha > 0$ and $\beta = -1$, then $\ch^{-1}_1(F) = 1$. Let $\ch_{\leq 2}^{-1}(F) = (r,1,d)$. If $r \leq 0$, we replace $F$ by $E/F$ in the following argument. Therefore, we may assume $r \geq 1$. If $r = 1$, we are dealing with the vertical wall, but that is located at $\beta = 0$. Thus, $r \geq 2$. A straightforward computation shows that the wall occurs for
\[
\alpha^2 = \frac{2d + 2}{r - 1}.
\]
This means $d > -1$. Furthermore, $\Delta(F) \geq 0$ implies
\[
d \leq \frac{1}{2r} < \frac{1}{2}.
\]
Since $d \in \tfrac{1}{2} + \Z$, we have $d = -\tfrac{1}{2}$. Then $\Delta(E/F) \geq 0$ implies $r = 2$. Overall, we have showed $\ch_{\leq 2}(F) = (2, -1, -\tfrac{1}{2})$ and by Lemma \ref{lem:2-1-1/2_p3} we know $\ch_3(F) \leq \tfrac{5}{6}$. Since $\ch_{\leq 2}(E/F) = (0,1,-\tfrac{5}{2})$, Lemma \ref{lem:simple_cases} implies $\ch_3(E/F) \leq \tfrac{19}{6}$. Overall, this means $e \leq 4$. Moreover, in case of equality we have $\ch_3(F) = \tfrac{5}{6}$, and $\ch_3(E/F) = \tfrac{19}{6}$. By Lemma \ref{lem:simple_cases} we get $E/F = \OO_V(-2)$.

The first exact sequence in the statement is giving the smallest wall by a special case of \cite[Theorem 5.1]{Sch15:stability_threefolds}.
\end{proof}

This finishes the special cases. For the rest of the section, we deal with the induction step for the general case.

\begin{lem}
\label{lem:no_rank_3_walls}
Let $E \in \Coh(\P^3)$ be a tilt-semistable object with $\ch(E) = (2,c,d,e)$. Assume that either
\begin{enumerate}
\item $c = -1$, $d \leq -\tfrac{3}{2}$, and $e \geq \tfrac{d^2}{2} - d + \tfrac{5}{24}$, or
\item $c = 0$, $d \leq -4$, and $e \geq \frac{d^2}{2} + \frac{d}{2} + 1$.
\end{enumerate}
Then $E$ is destabilized along a semicircular wall by a subobject or quotient of rank at most two. 
\end{lem}

\begin{proof}
\begin{enumerate}
\item Assume $c = -1$, $d \leq -\tfrac{3}{2}$, and $e \geq \tfrac{d^2}{2} - d + \tfrac{5}{24}$. Then the radius  of $Q_{\alpha, \beta}(E) = 0$ is bounded from below by
\[
\rho_Q^2 \geq \frac{144d^4 - 32d^3 + 24d^2 - 24d + 5}{16(4d - 1)^2}.
\]
We can compute
\[
\rho_Q^2 - \frac{\Delta(E)}{12} \geq \frac{(108d^2 - 68d + 11)(2d + 1)^2}{48(4d - 1)^2} > 0.
\]
We conclude by Lemma \ref{lem:higherRankBound}.
\item Assume $c = 0$, $d \leq -4$, and $e \geq \frac{d^2}{2} + \frac{d}{2} + 1$. Then the radius  of $Q_{\alpha, \beta}(E) = 0$ is bounded from below by
\[
\rho_Q^2 \geq \frac{9d^4 + 34d^3 + 45d^2 + 36d + 36}{16d^2}. 
\]
We can compute
\[
\rho_Q^2 - \frac{\Delta(E)}{12} \geq \frac{(27d^3 + 37d^2 + 24d + 36)(d + 3)}{48d^2} > 0. 
\]
We conclude by Lemma \ref{lem:higherRankBound}. \qedhere
\end{enumerate}
\end{proof}

\begin{lem}
\label{lem:rank_one_walls}
Let $E \in \Coh(\P^3)$ be a tilt-semistable object with $\ch(E) = (2,c,d,e)$.
\begin{enumerate}
\item Assume $c = -1$ and $d \leq -\tfrac{3}{2}$. If $E$ is destabilized by a subobject $F$ of rank one, then
\[
e \leq \frac{d^2}{2} - d + \frac{5}{24}.
\]
In case of equality, we have $F \cong \OO(-1)$ and $E/F$ is the ideal sheaf of a plane curve.
\item Assume $c = 0$ and $d \leq -4$. If $E$ is destabilized by a subobject or quotient $F$ of rank one, then
\[
e < \frac{d^2}{2} + \frac{d}{2} + 1.
\]
\end{enumerate}
\end{lem}

\begin{proof}
\begin{enumerate}
\item Assume $c = -1$ and $e \geq \tfrac{d^2}{2} - d + \tfrac{5}{24}$. We can compute
\[
Q_{0, -2}(E) \leq 4d^2 - 20d - 18e + 4 \leq -\frac{(10d - 1)(2d + 1)}{4} < 0.
\]
Thus, any wall destabilizing $E$ must contain a point $(\alpha, -2)$. In particular,
\[
0 < \ch_1^{-2}(F) = \ch_1(F) + 2 < \ch_1^{-2}(E) = 3.
\]
This implies $\ch_1(F) \in \{-1, 0\}$. For proving the bound, we can assume $\ch_1(F) = -1$ and $\ch(F) = (1, 0, -y, z) \cdot \ch(\OO(-1))$ for some $y \geq 0$. If $\ch_1(F) = 0$, the same argument will work when $F$ is replaced by $E/F$. By Proposition \ref{prop:rank_one_old} we know
\[
z \leq \frac{y(y+1)}{2}.
\]
We can compute
\begin{align*}
s_Q(E) &= \frac{d + 6e}{4d - 1} \leq \frac{12d^2 - 20d + 5}{16d - 4}, \\
s(E, F) & = d + 2y - 1.
\end{align*}
Since $s(E, F) \leq s_Q(E)$, we have
\[
y \leq \frac{4d^2 - 1}{8 - 32d} < -\frac{d}{2} - \frac{1}{4}.
\]
Using Proposition \ref{prop:rank_one_old} on the quotient $E/F$ leads to
\[
e \leq \frac{d^2}{2} + dy + \frac{y^2}{2} - d + z + \frac{5}{24} \leq \frac{d^2}{2} + dy + y^2 - d + \frac{y}{2} + \frac{5}{24}.
\]
This is a parabola in $y$ with minimum at $y = -\tfrac{d}{2} - \tfrac{1}{4}$. This means the maximum occurs at $y = 0$, where we get
\[
e \leq \frac{d^2}{2} - d + \frac{5}{24}.
\]
In case of equality, we must have $\ch(F) = \ch(\OO(-1))$. By Proposition \ref{prop:line_bundles_uniquely_stable}, we get $F \cong \OO(-1)$. Proposition \ref{prop:rank_one_old} implies that $E/F$ is the ideal sheaf of a plane curve.

If instead $\ch_0(F) = 0$, then $E/F \cong \OO(-1)$ and $F \cong \II_C$ for a plane curve $C$. Then $\Ext^1(\OO(-1), \II_C) = 0$ implies that $E$ is just a direct sum.
\item Assume $c = 0$ and $e \geq  \tfrac{d^2}{2} + \tfrac{d}{2} + 1$. We can compute
\[
Q_{0, -1}(E) \leq 4d^2 - 4d - 12e \leq -2(d+2)(d+3) < 0.
\]
Thus, any wall destabilizing $E$ must contain a point $(\alpha, -1)$. In particular,
\[
0 < \ch_1^{-1}(F) = \ch_1(F) + 1 < \ch_1^{-1}(E) = 2.
\]
This means $\ch_1(F) = 0$, a contradiction to the fact that we are not dealing with the vertical wall. \qedhere
\end{enumerate}
\end{proof}

\begin{proof}[Proof of Theorem \ref{thm:sequences_maximal_ch3}]
The proof will be by induction on $\Delta(E)$. The start of the induction is done by Lemma \ref{lem:200_p3}, \ref{lem:2-1-1/2_p3}, \ref{lem:20-1_p3}, \ref{lem:20-2_p3}, and \ref{lem:20-3_p3}. The induction step will be by contradiction. The strategy is to show that there is no wall outside the semidisk $Q_{\alpha, \beta}(E) < 0$ and therefore, there is no wall for such an object unless we have equality in the claimed bound. By Lemma \ref{lem:no_rank_3_walls} we are able to infer that $E$ is destabilized along a semicircular wall $W$ induced by an exact sequence $0 \to F \to E \to G \to 0$, where $F$ is of rank $r \in \{0, 1, 2\}$. Note that $\Delta(F) < \Delta(E)$, and we intend to use the induction hypothesis on $F$ in case $r = 2$.
\begin{enumerate}
\item Assume $c = -1$, $d \leq -\tfrac{3}{2}$, and $e \geq \tfrac{d^2}{2} - d + \tfrac{5}{24}$.
\begin{itemize}
\item If $r = 1$, we can use Lemma \ref{lem:rank_one_walls} to get that $F = \OO(-1)$, and $G = \II_C$ for a plane curve $C$. Then there is a map $\OO(-1) \to \II_C$. An application of the Snake Lemma shows that there is an injective morphism $\OO(-1)^{\oplus 2} \into E$. This reduces to the case of rank two walls.
\item Up to exchanging $F$ and $G$ we can assume $r = 2$. At the end we will show that $F$ is indeed the subobject and not the quotient. As in the proof of Lemma \ref{lem:rank_one_walls}, we know that $Q_{0,-2}(E) < 0$. Moreover, we can compute
\[
Q_{0,-\tfrac{3}{2}}(E) = 4d^2 - 12d - 12e + \frac{9}{4} \leq -2d^2 - \frac{1}{4} < 0.
\]
Both $\ch_1^{-3/2}(F) > 0$ and $\ch_1^{-2}(F) < \ch_1^{-2}(E)$ together imply $\ch_1(F) = -2$. Therefore, we may assume that $\ch(F) = (2,0,y,z) \cdot \ch(\OO(-1))$ for some $y \leq 0$. If $y = 0$, then $z \leq 0$. Applying Lemma \ref{lem:simple_cases} to the quotient $G$ implies that $e \leq \tfrac{d^2}{2} - d + \tfrac{5}{24}$ with equality if and only if $G \cong \OO_V(d - \tfrac{1}{2})$ for a plane $V \subset \P^3$. Assume for a contradiction that $y \leq -1$. By induction we know that
\[
z \leq \frac{y^2}{2} + \frac{y}{2} + 1.
\]
We can compute 
\begin{align*}
s_Q(E) &= \frac{d + 6e}{4d - 1} \leq \frac{12d^2 - 20d + 5}{16d - 4}, \\
s(E, F) & = d - y - 1.
\end{align*}
Since $s(E, F) \leq s_Q(E)$, we have
\[
y \geq \frac{4d^2 - 1}{16d - 4} > \frac{d}{2} - \frac{1}{4}.
\]
If $d = -\tfrac{3}{2}$, then this means $y > -1$, a contradiction. Thus, we may assume $d \leq -\tfrac{5}{2}$. Using Theorem \ref{thm:rank_zero_bound} on the quotient $G$ leads to
\[
e \leq \frac{d^2}{2} - dy + \frac{y^2}{2} - d + z + \frac{5}{24} \leq \frac{d^2}{2} - dy + y^2 - d + \frac{y}{2} + \frac{29}{24}.
\]
This is a parabola in $y$ with minimum at $y = \tfrac{d}{2} - \tfrac{1}{4}$. This means the maximum occurs at $y = -1$, where we get
\[
e \leq \frac{d^2}{2} + \frac{41}{24} < \frac{d^2}{2} - d + \frac{5}{24}.
\]
Overall, we showed that the only case in which we can get $e = \tfrac{d^2}{2} - d + \tfrac{5}{24}$ is when $\ch(F) = 2\ch(\OO(-1))$. We can conclude by Proposition \ref{prop:line_bundles_uniquely_stable} that $F = \OO(-1)^{\oplus 2}$. If $\OO(-1)^{\oplus 2}$ is the quotient and not the subobject, then $\Ext^1(\OO(-1), \OO_V(d - \tfrac{1}{2})) = 0$ shows that $E$ is simply a direct sum.
\end{itemize}
\item Assume $c = 0$, $d \leq -4$, and $e \geq \tfrac{d^2}{2} + \tfrac{d}{2} + 1$. By Lemma \ref{lem:rank_one_walls}, we know that either $F$ or $G$ has to have rank two. Most of the argument is numerical, and for the moment we assume that $F$ has rank two. We will argue at the end that $F$ is indeed the subobject and not the quotient. As in the proof of Lemma \ref{lem:rank_one_walls}, we get $Q_{0,-1}(E) < 0$. Thus,
\[
0 < \ch_1^{-1}(F) = \ch_1(F) + 2 < \ch_1^{-1}(E) = 2.
\]
This means we can assume that $\ch(F) = (2,-1,y,z)$ for some $y \leq -\tfrac{1}{2}$. By induction, we know that
\[
z \leq \frac{y^2}{2} - y + \frac{5}{24}.
\]
We can compute
\begin{align*}
s_Q(E) &= \frac{3e}{2d} \geq \frac{3d^2 + 3d + 6}{4d}, \\
s(E, F) & = d - y.
\end{align*}
Since $s(E, F) \leq s_Q(E)$, we have
\[
y \geq \frac{d^2 - 3d - 6}{4d} > \frac{d}{2} + \frac{1}{2}.
\]
Using Theorem \ref{thm:rank_zero_bound} on the quotient $G$ leads to
\[
e \leq \frac{d^2}{2} - dy + \frac{y^2}{2} + z + \frac{1}{24} \leq \frac{d^2}{2} - dy + y^2 - y + \frac{1}{4}.
\] 
This is a parabola in $y$ with minimum at $y = \tfrac{d}{2} + \tfrac{1}{2}$. This means the maximum occurs at $y = -\tfrac{1}{2}$, where we get
\[
e \leq \frac{d^2}{2} + \frac{d}{2} + 1.
\]
Moreover, equality happens when $F \in M(2,-1,-\tfrac{1}{2},\tfrac{5}{6})$. By Lemma \ref{lem:simple_cases} we have $G \cong \OO_V(d + 1)$. We are left to show that $F$ is indeed a subobject. As before, the strategy will be to show that $\Ext^1(F, \OO_V(d + 1))$ vanishes. By Lemma \ref{lem:2-1-1/2_p3} we have a short exact sequence of sheaves.
\[
0 \to \OO(-2) \to \OO(-1)^{\oplus 3} \to F \to 0.
\]
The long exact sequence from applying the functor $\Hom( \cdot, \OO_V(d + 1))$ to this sequence immediately concludes the proof. \qedhere
\end{enumerate}
\end{proof}


\section{Geometric structure of the moduli spaces}

From the classification in the last section, we can deduce a geometric description of their moduli spaces.

\begin{cor}
\label{cor:moduli spaces}
\begin{enumerate}
\item We have $M(2,-1,-\tfrac{1}{2},\tfrac{5}{6}) \cong \P^3$ and $M(2,0,-1,0) \cong \P^5$.
\item The moduli space $M(2,0,-3,4)$ is the blow up of $\Gr(3,10)$ in a smooth subvariety isomorphic to $\P^3 \times \P^3$.
\item For $d \leq -\tfrac{3}{2}$ the moduli space $M(2,-1,d,\tfrac{d^2}{2} - d + \tfrac{5}{24})$ is a $\Gr(2,n)$-bundle over $\P^3$, where
\[
n = \binom{\frac{5}{2} - d}{2}. 
\]
\item For $d \leq -4$ the moduli space $M(2,0,d,\tfrac{d^2}{2} + \tfrac{d}{2} + 1)$ is a $\P^n$-bundle over the product $\P^3 \times \P^3$, where $n = d(d-2) - 1$.
\end{enumerate}
\end{cor}

In order to proof this statement we need to recall some notation and known results. Let $f: X \to Y$ be a morphism between projective varieties, and let $F, G \in \Coh(X)$. For any $i \in \Z$ the relative Ext-sheaf is defined to be
\[
\lExt^i_f(F, G) := \mathbf{R}^i(f_* \lHom(F, \cdot))(G) = \HH^i(Rf_* \RlHom(F, G)).
\]
In \cite{Lan83:families_extensions} Lange constructs universal families of extensions of sheaves using these relative Ext-sheaves. However, the case of the Grassmann bundle requires a few steps beyond what Lange did.

\begin{thm}[{\cite[Theorem 1.4]{Lan83:families_extensions}}]
\label{thm:lange_base_change}
Let $y \in Y$ and assume that the base change morphism $\tau^i(y): \lExt^i_f(F, G) \otimes_Y \C(y) \to \Ext^i_{X_y}(F_y, G_y)$ is surjective.
\begin{enumerate}
    \item There is a neighborhood $U$ of $y$ such that the base change morphism $\tau^i(y')$ is an isomorphism for all $y' \in U$.
    \item The base change morphism $\tau^{i-1}(y)$ is surjective if and only if $\lExt^i_f(F, G)$ is locally free in a neighborhood of $y$.
\end{enumerate}
\end{thm}

Note that Grauert's Theorem \cite[Corollary III.12.9]{Har77:algebraic_geometry} shows that if $\tau^i(y)$ is an isomorphism for all $y \in Y$ and the dimension of $\Ext^i_{X_y}(F_y, G_y)$ is independent of $y$, then $\lExt^i_f(F, G)$ is locally free. Therefore, Lange's theorem creates opportunities for descending induction on $i$.

\subsection{The case $c = -1$, $d = -\tfrac{1}{2}$}

The moduli space of three-dimensional quotients of $H^0(\OO(1))^{\vee}$ is given by $\P^3$, and let $M := M(2,-1,-\tfrac{1}{2},\tfrac{5}{6})$. We can define a function $\varphi: \P^3 \to M$ as follows. If $H^0(\OO(1))^{\vee} \onto U$ is a three-dimensional quotient, then we get a short exact sequence of sheaves
\[
0 \to \OO(-2) \to \OO(-1) \otimes U \to E \to 0.
\]
We set $\varphi(U) = E$. We will have to proof the following lemma.

\begin{lem}
\label{lem:2-1-1/2_moduli}
\begin{enumerate}
    \item The function $\varphi$ is well-defined, i.e., $E$ is slope-stable.
    \item The function $\varphi$ is bijective
    \item The function $\varphi$ is a morphism of schemes.
    \item The moduli space $M$ is smooth, and therefore, $\varphi$ is an isomorphism.
\end{enumerate}
\end{lem}

\begin{proof}
\begin{enumerate}
    \item By Theorem \ref{thm:sequences_maximal_ch3}, we know that there is only one wall for such objects $E$ given by a sequence
    \[
    0 \to \OO(-1)^{\oplus 3} \to E \to \OO(-2)[1] \to 0.
    \]
    Therefore, showing that $E = \varphi(U)$ is slope-stable is the same as showing that it is $\nu_{\alpha, \beta}$-semistable in a neighborhood above this wall $W$. If it is not semistable above $W$, then there is a destabilizing semistable quotient $E \onto G$. Clearly, $E$ is strictly-semistable along $W$. Therefore, such a quotient must satisfy $\nu_{\alpha, \beta}(E) = \nu_{\alpha, \beta}(G)$ for $(\alpha, \beta)$ along $W$. Since $E$ has the JH-property, we know that in any Jordan-H\"older filtration of $E$ there are three stable factors $\OO(-1)$ and one stable factor $\OO(-2)[1]$. This means that the stable factors of $G$ have to be a subset of these.
    
    A quotient $\OO(-2)[1]$ does not destabilizes $E$ above $W$ for purely numerical reasons. The vector space $\Hom(E, \OO(-1))$ is the kernel of the morphism $\Hom(\OO(-1) \otimes U, \OO(-1)) \to \Hom(\OO(-2), \OO(-1))$ which is injective. Thus, $G \not \cong \OO(-1)^{\oplus a}$ for $a \in \{1, 2, 3\}$. If $G$ is an extension between $\OO(-1)$ and $\OO(-2)[1]$, then the kernel is $\OO(-1)^{\oplus 2}$ and this does not destabilize $E$ above the wall. Similarly, if $G$ is an extension between $\OO(-1)^{\oplus 2}$ and $\OO(-2)[1]$, then the kernel is given by $\OO(-1)$ which does not destabilize $E$ above the wall.
    
    \item By Theorem \ref{thm:sequences_maximal_ch3} we know that any semistable $E$ fits into an exact sequence
    \[
    0 \to \OO(-1)^{\oplus 3} \to E \to \OO(-2)[1] \to 0.
    \]
    Giving such an extension is the same as giving an element in $\Ext^1(\OO(-2)[1], \OO(-1)^{\oplus 3}) = H^0(\OO(1))^{\oplus 3}$. In Theorem \ref{thm:sequences_maximal_ch3} we have already shown that this is the Harder-Narasimhan filtration of $E$ below the wall. The Harder-Narasimhan factors are unique. This means $E$ determines the subobject $\OO(-1)^{\oplus 3}$. However, the group $\GL(3)$ acts via automorphisms on $\OO(-1)^{\oplus 3}$ without changing the isomorphism class of $E$. This means we get a unique subspace of $H^0(\OO(1))^{\oplus 3}$. However, if this subspace is not of dimension three, then there is a destabilizing morphism $E \onto \OO(-1)$. This proves both surjectivity and injectivity.
    
    \item We will construct a family on $\P^3 \times \P^3$ whose fibers are in bijection with objects in $M$. The universal property of $M$ then shows that $\varphi$ is a morphism. We have two projections
    \[
    p_1, p_2: \P^3 \times \P^3 \to \P^3.
    \]
    Let $\OO \otimes H^0(\OO(1))^{\vee} \onto Q$ be the universal rank three quotient bundle whose fibers parametrize three-dimensional quotients $H^0(\OO(1))^{\vee} \onto U$. We can compose the morphisms
    $p_2^*\OO(-2) \to p_2^* \OO(-1) \otimes H^0(\OO(1))^{\vee} \to p_2^* \OO(-1) \otimes p_1^* Q$. Taking the quotient leads to an object $\UU$. By construction this is the desired family.
    
    \item In order to show that $M$ is smooth all we have to do is to show that $\Ext^2(E, E) = 0$. However, applying the three functors $\RHom( \cdot , \OO(-2)[1])$, $\RHom( \cdot, \OO(-1))$, and $\RHom(E, \cdot)$ to
    \[
    0 \to \OO(-1)^{\oplus 3} \to E \to \OO(-2)[1] \to 0
    \]
    implies this immediately. \qedhere
\end{enumerate}
\end{proof}

\subsection{The case $c = 0$, $d = -1$}

Note that $\Hom(T(-3), \OO(-1)) \cong \C^6$. Let $Q$ be the generalized Kronecker quiver with two vertices and six arrows between them, all going in the same direction. By Lemma \ref{lem:20-1_p3} and Theorem \ref{thm:macri_p3} we know that $M(2,0,-1,0)$ is isomorphic to the moduli space of quiver representations of $Q$ with dimension vector $(1, 5)$. This space parametrizes six vectors in $\C^5$ modulo the action of $\GL(5)$. It is not hard to see that this space is $\P^5$.

\subsection{The case $c = 0$, $d = -3$}

This is the only case in which there is more than one chamber where the moduli space of tilt-semistable objects is non-trivial. By Theorem \ref{thm:sequences_maximal_ch3} there are exactly two walls in tilt stability for objects with Chern character $v = (2,0,-3,4)$. The two walls are
\begin{align*}
    W_1 &= W\left(v, \left(2,-1,-\frac{1}{2}\right)\right), \\
    W_2 &= W(v, \OO(-1)).
\end{align*}
Note that $W_2$ is located inside $W_1$. By Theorem \ref{thm:sequences_maximal_ch3}, there are no semistable objects inside $W_2$. Let $M'$ be the moduli space of tilt-semistable objects in between $W_1$ and $W_2$. Note that $M'$ does not parametrize any strictly semistable objects. The first goal is to show that $M'$ is isomorphic to $\Gr(10, 3)$ the moduli space of three-dimensional quotients of $H^0(\OO(2))^{\vee}$. We define a function $\varphi: \Gr(10, 3) \to M'$ as follows. If $H^0(\OO(2))^{\vee} \to U$ is a three-dimensional quotient, then we get a short exact sequence of sheaves
\[
0 \to \OO(-3) \to \OO(-1) \otimes U \to E \to 0.
\]
We set $\varphi(U) = E$.

\begin{lem}
\begin{enumerate}
    \item The function $\varphi$ is well defined, i.e., $E$ is Gieseker-stable.
    \item The function $\varphi$ is bijective.
    \item The function $\varphi$ is a morphism of schemes.
    \item The moduli space $M'$ is smooth, and therefore, $\varphi$ is an isomorphism.
\end{enumerate}
\end{lem}

\begin{proof}
The proof is essentially the same as that of Lemma \ref{lem:2-1-1/2_moduli} with $\OO(-2)$ replaced by $\OO(-3)$.
\end{proof}

Next, we have to understand crossing the wall $W_1$ to describe $M(2,0,-3,4)$. 

\begin{lem}
\label{lem:ext_special_case_d3}
Let $F \in M(2,0,-\tfrac{1}{2}, \tfrac{5}{6})$ and $V \subset \P^3$ be a plane. Then
\begin{align*}
    \ext^i(F,F) &= \begin{cases}
    1 &\text{, if $i = 0$} \\
    3 &\text{, if $i = 1$} \\
    0 &\text{, otherwise}
    \end{cases}, \\
    \ext^i(\OO_V(-2), \OO_V(-2)) &= \begin{cases}
    1 &\text{, if $i = 0$} \\
    3 &\text{, if $i = 1$} \\
    0 &\text{, otherwise}
    \end{cases}, \\
    \ext^i(F, \OO_V(-2)) &= \begin{cases}
    1 &\text{, if $i = 1$} \\
    0 &\text{, otherwise}
    \end{cases}, \\
    \ext^i(\OO_V(-2), F) &= \begin{cases}
    15 &\text{, if $i = 1$} \\
    0 &\text{, otherwise}
    \end{cases}.
\end{align*}
\end{lem}

\begin{proof}
By Theorem \ref{thm:sequences_maximal_ch3} and Lemma \ref{lem:2-1-1/2_moduli} there is an exact sequence of sheaves
\[
0 \to \OO(-2) \to \OO(-1)^{\oplus 3} \to F \to 0,
\]
where the three linear polynomials defining the first map are linearly independent. The sequence
\[
0 \to \OO(-1) \to \OO \to \OO_V \to 0
\]
shows that the derived dual of $\OO_V$ is given by $\OO_V(1)$. From here the statement is a straightforward computation involving the appropriate long exact sequences.
\end{proof}

By Theorem \ref{thm:sequences_maximal_ch3} any tilt-stable objects $E$ that is destabilized by either $0 \to F \to E \to \OO_V(-2) \to 0$ or $0 \to \OO_V(-2) \to E \to F \to 0$ satisfies the JH-property along $W_1$. This means all non-trivial extensions in $\Ext^1(\OO_V(-2), F)$ or $\Ext^1(F, \OO_V(-2))$ are stable on one side of the wall.

\begin{lem}
\label{lem:destab_locus}
The closed subscheme of $M'$ parametrizing objects $E$ fitting into a sequence
\[
0 \to \OO_V(-2) \to E \to F \to 0
\]
is isomorphic to $\P^3 \times \P^3$.
\end{lem}

\begin{proof}
We have showed that the moduli space $M(2,-1,-\tfrac{1}{2},\tfrac{5}{6})$ parametrizes three-dimensional subspaces $U \subset H^0(\OO(1))$. Moreover, the space $M' \cong \Gr(3, 10)$ parametrizes three-dimensional subspaces $W \subset H^0(\OO(2))$. Any plane in $\P^3$ is cut out by a linear equation $l$. From this we get a closed embedding
\[
\P^3 \times \P^3 \cong M(2,-1,-\tfrac{1}{2},\tfrac{5}{6}) \times \Gr(3,4) \into \Gr(3, 10)
\]
as follows. If $U \subset H^0(\OO(1))$ with $\dim U = 3$ and a $l$ is a linear equation cutting out a plane in $\P^3$, then we get a three-dimensional subspace $l \cdot U \subset H^0(\OO(2))$. The goal in this argument is to show that this image is precisely the locus in $M'$ that is destabilized at the wall $W_1$.

Let $W = l \cdot U \subset H^0(\OO(2))$ be as above. Then we get a short exact sequence
\[
0 \to \OO(-3) \to \OO(-1) \otimes U \to E \to 0.
\]
The morphism $\OO(-3) \to \OO(-1) \otimes W$ factors through $\OO(-2) \to \OO(-1) \otimes U$ whose quotient is an element $F \in M(2,-1,-\tfrac{1}{2},\tfrac{5}{6})$. By the Snake Lemma the kernel of $E \onto F$ is given by $\OO_V(-2)$, where $V$ is cut out by $l$.

Assume vice versa that there is a short exact sequence
\[
0 \to \OO_V(-2) \to E \to F \to 0.
\]
Then $V$ is cut out by a linear equation $l$. Since there is also a short exact sequence
\[
0 \to \OO(-3) \to \OO(-1)^{\oplus 3} \to E \to 0,
\]
we get a morphism $\OO(-1)^{\oplus 3} \onto F$ whose kernel has to be $\OO(-2)$. This morphism $\OO(-2) \to \OO(-1)^{\oplus 3}$ gives a three-dimensional subspace $U \subset H^0(\OO(1))$. By construction the subspace $W = l \cdot U \subset H^0(\OO(2))$ represents $E$.
\end{proof}

We need the following classical result by Moishezon. Recall that the analytification of a smooth proper algebraic spaces of finite type over $\C$ of dimension $n$ is a complex manifold with $n$ independent meromorphic functions. Moishezon's result is originally stated in these terms as his work predated algebraic spaces.

\begin{thm}[\cite{Moi67:moishezon_manifolds}]
\label{thm:blow_up}
Any birational morphism $f: X \to Y$ between smooth proper algebraic spaces of finite type over $\C$ such that the contracted locus $E$ is irreducible and the image $f(E)$ is smooth is the blow up of $Y$ in $f(E)$.
\end{thm}

Since $\Ext^1(F, \OO_V(-2)) = \C$ independently of $F$ and $V$, there is a unique stable extension $0 \to \OO_V(-2) \to E \to F \to 0$ for each $F$ and $V$. Therefore, we get a morphism $M(2,0,-3,4) \to M'$ which is birational outside of the objects destabilized by this type of sequence. By Lemma \ref{lem:destab_locus} the exceptional locus maps onto a smooth projective subvariety. The fibers are all irreducible and given by $\P(\Ext^1(\OO_V(-2), F)) \cong \P^{14}$. Therefore, the exceptional locus is irreducible. Theorem \ref{thm:blow_up} concludes the proof together with the following lemma.

\begin{lem}
The moduli space $M(2,0,-3,4)$ is smooth.
\end{lem}

\begin{proof}
Applying the three functors $\RHom( \cdot, F)$, $\RHom( \cdot, \OO_V(-2))$, and $\RHom(E, \cdot)$ to
\[
0 \to F \to E \to \OO_V(-2) \to 0
\]
leads to
\[
\Ext^2(E, E) = 0. \qedhere
\]
\end{proof}

\subsection{The general case with $c = -1$}

\begin{lem}
Let $d \leq -\tfrac{3}{2}$, and let $V \subset \P^3$ be plane. Giving a slope-stable sheaf $E$ that can be written as an extension
\[
0 \to \OO(-1)^{\oplus 2} \to E \to \OO_V\left( d - \frac{1}{2} \right) \to 0
\]
is equivalent to giving a subspace of $\Ext^1(\OO_V(d - \tfrac{1}{2}), \OO(-1))$ of dimension two.
\end{lem}

\begin{proof}
Giving such an extension is the same as giving an element in
\[
\Ext^1\left(\OO_V\left(d - \frac{1}{2}\right), \OO(-1)^{\oplus 2}\right).
\]
In Theorem \ref{thm:sequences_maximal_ch3} we have already shown that for semistable objects this short exact sequence is the Harder-Narasimhan filtration of $E$ once $E$ becomes tilt-unstable. The Harder-Narasimhan factors are unique. This means $E$ determines both $V$ and the subobject $\OO(-1)^{\oplus 2}$. However, the group $\GL(2)$ acts via automorphisms on $\OO(-1)^{\oplus 2}$ without changing the isomorphism class of $E$. This means we get a subspace of  $\Ext^1(\OO_V(d - \tfrac{1}{2}), \OO(-1))$.

If this subspace is of dimension zero, then $E$ is a direct sum and certainly unstable. Assume the subspace is of dimension one. Then the morphism $\OO_V(d - \frac{1}{2}) \to \OO(-1)^{\oplus 2}[1]$ factors through $\OO(-1)[1]$. The octahedron axiom implies that there is a map $E \onto \OO(-1)$ in contradiction to stability.

Assume vice versa that we have a two dimensional subspace of  $\Ext^1(\OO_V(d - \tfrac{1}{2}), \OO(-1))$. Choosing two arbitrary basis elements leads to an extension
\[
0 \to \OO(-1)^{\oplus 2} \to E \to \OO_V\left( d - \frac{1}{2} \right) \to 0.
\]
This object $E$ is strictly semistable along the induced wall $W$. Its Jordan-H\"older factors along the wall are two copies of $\OO(-1)$ and one copy of $\OO_V(d - \tfrac{1}{2})$. The Jordan-H\"older factors of any destabilizing subobject must be a subset of these. By construction we have $\Hom(E, \OO(-1)) = 0$. Therefore, neither $\OO_V(d - \tfrac{1}{2})$ nor an extension between $\OO_V(d - \tfrac{1}{2})$ and $\OO(-1)$ can be a subobject of $E$.
\end{proof}

The argument will proceed in three steps. First we construct the Grassmann bundle that we expect to be the moduli space. Then we construct a global family on this space. This family will induce a morphism, and we finish by showing that it is an isomorphism.

Let $\VV \subset \Gr(3,4) \times \P^3 \cong \P^3 \times \P^3$ be the universal plane. There are two projections $p: \Gr(3,4) \times \P^3 \to \Gr(3,4)$ and $q: \Gr(3,4) \times \P^3 \to \P^3$. The dimension of the group $\Ext^i(\OO_V(d - \tfrac{1}{2}), \OO(-1)) = H^{i - 1}(\OO_V(\tfrac{1}{2} - d))$ is independent of the plane $V \subset \P^3$ and non-zero if and only if $i \neq 1$. By Theorem \ref{thm:lange_base_change} this implies that
\[
A := \lExt^1_p(\OO_{\VV} \otimes q^* \OO(d - \tfrac{1}{2}), q^* \OO(-1)) \cong Rp_* \RHom(\OO_{\VV} \otimes q^* \OO(d - \tfrac{1}{2}), q^* \OO(-1))[1] 
\]
is a vector bundle such that the natural map $A_{V} \to \Ext^1(\OO_V(d - \tfrac{1}{2}), \OO(-1))$ is an isomorphism for every plane $V \subset \P^3$.

Let $\Gr(A^{\vee}, 2)$ be the Grassmann bundle parametrizing locally free rank two quotients of $A^{\vee}$. There is a projection $\pi: \Gr(A^{\vee}, 2) \to \Gr(3,4)$. Let the quotient $\pi^* A^{\vee} \onto \QQ$ be the universal quotient bundle.

Let $E$ be a stable sheaf as above. Then we get a commutative diagram with exact rows:

\centerline{
\xymatrix{
0 \ar[r] & \OO(-1) \otimes \Ext^1\left(\OO_V\left( d - \frac{1}{2} \right), \OO(-1)\right)^{\vee} \ar[r] \ar[d] & E_V \ar[r] \ar[d] & \OO_V\left( d - \frac{1}{2} \right) \ar[r] \ar@{=}[d] & 0 \\
0 \ar[r] & \OO(-1)^{\oplus 2} \ar[r] & E \ar[r] & \OO_V\left( d - \frac{1}{2} \right) \ar[r] & 0.
}}

Here the top row is induced by the natural morphism
\[
\OO_V\left( d - \frac{1}{2} \right) \to \OO(-1)[1] \otimes \Hom\left(\OO_V\left( d - \frac{1}{2} \right), \OO(-1)[1] \right)^{\vee}.
\]
We will globalize this diagram to obtain a family. We can compute
\begin{align*}
	\Hom(A, A) &= \Hom\left(p^* A, \RHom(\OO_{\VV} \otimes q^* \OO(d - \tfrac{1}{2}), q^* \OO(-1))[1] \right) \\
	&= \Hom\left(p^* A \otimes \OO_{\VV} \otimes q^* \OO\left( d - \frac{1}{2} \right), q^* \OO(-1)[1] \right) \\
	&= \Hom\left(\OO_{\VV} \otimes q^* \OO\left( d - \frac{1}{2} \right), q^* \OO(-1) \otimes p^* A^{\vee}[1] \right).
\end{align*}

Choosing the identity in this group leads to an extension
\[
0 \to q^* \OO(-1) \otimes p^* A^{\vee} \to \WW \to \OO_{\VV} \otimes q^* \OO\left( d - \frac{1}{2} \right) \to 0,
\]
whose restriction to each plane $V \subset \P^3$ is $E_V$. Let $\tilde{p}: \Gr(A^{\vee}, 2) \times \P^3 \to \Gr(A^{\vee}, 2)$ be the first projection and let $\tilde{q}: \Gr(A^{\vee}, 2) \times \P^3 \to \P^3$ be the second projection. We get a commutative diagram with exact rows:

\centerline{
\xymatrix{
0 \ar[r] & \tilde{q}^* \OO(-1) \otimes \tilde{p}^* \pi^* A^{\vee} \ar[r] \ar[d] & \WW \ar[r] \ar[d] & (\pi \times \id)^* \left( \OO_{\VV} \otimes q^* \OO\left( d - \frac{1}{2} \right) \right) \ar[r] \ar@{=}[d] & 0 \\
0 \ar[r] & \tilde{q}^* \OO(-1) \otimes \tilde{p}^* \QQ \ar[r] & \UU \ar[r] & (\pi \times \id)^* \left( \OO_{\VV} \otimes q^* \OO\left( d - \frac{1}{2} \right) \right) \ar[r] & 0.
}}

Here $\UU$ is a family of stable objects with Chern character $(2,-1,d,\tfrac{d^2}{2} - d + \tfrac{5}{24})$ living in $\Gr(A^{\vee}, 2)$ that induces a bijective morphism $\Gr(A^{\vee}, 2)) \to  M(2,-1,d,\tfrac{d^2}{2} - d + \tfrac{5}{24})$. Since we are in characteristic zero, the following lemma will finish the argument.

\begin{lem}
The moduli space $M(2,-1,d,\tfrac{d^2}{2} - d + \tfrac{5}{24})$ is smooth.
\end{lem}

\begin{proof}
Applying the three functors $\RHom( \cdot , \OO(-1))$, $\RHom( \cdot, \OO_V(d - \tfrac{1}{2}))$, and $\RHom(E, \cdot)$ to
\[
0 \to \OO(-1)^{\oplus 2} \to E \to \OO_V\left(d - \frac{1}{2}\right) \to 0
\]
leads to
\[
\ext^1(E, E) = 2 \ext^1(\OO_V(d - \tfrac{1}{2}), \OO(-1)) - 1 = \dim \Gr(A^{\vee}, 2). \qedhere
\]
\end{proof}

\subsection{The general case with $c = 0$}

\begin{lem}
\label{lem:c0_general_ext}
Let $F \in M(2,-1,-\tfrac{1}{2},\tfrac{5}{6})$, $V \subset \P^3$, and $d \leq -4$. Then
\begin{align*}
    \ext^i(F,F) &= \begin{cases}
    1 &\text{, if $i = 0$} \\
    3 &\text{, if $i = 1$} \\
    0 &\text{, otherwise}
    \end{cases}, \\
    \ext^i(\OO_V(-2), \OO_V(-2)) &= \begin{cases}
    1 &\text{, if $i = 0$} \\
    3 &\text{, if $i = 1$} \\
    0 &\text{, otherwise}
    \end{cases}, \\
    \ext^i(F, \OO_V(d + 1)) &= \begin{cases}
    (d+4)(d+2) &\text{, if $i = 2$} \\
    0 &\text{, otherwise}
    \end{cases}, \\
    \ext^i(\OO_V(d + 1), F) &= \begin{cases}
    d(d-2) &\text{, if $i = 1$} \\
    0 &\text{, otherwise}
    \end{cases}.
\end{align*}
\end{lem}

\begin{proof}
By Theorem \ref{thm:sequences_maximal_ch3} and Lemma \ref{lem:2-1-1/2_moduli} there is an exact sequence of sheaves
\[
0 \to \OO(-2) \to \OO(-1)^{\oplus 3} \to F \to 0,
\]
where the three linear polynomials defining the first map are linearly independent. The sequence
\[
0 \to \OO(-1) \to \OO \to \OO_V \to 0
\]
shows that the derived dual of $\OO_V$ is given by $\OO_V(1)$. From here the statement is a straightforward computation involving the appropriate long exact sequences.
\end{proof}

\begin{lem}
Let $d \leq -4$, $V \subset \P^3$ be a plane, and $F \in M(2,-1,-\tfrac{1}{2},\tfrac{5}{6})$. Giving a $2$-Gieseker-stable sheaf $E$ that can be written as an extension
\[
0 \to F \to E \to \OO_V(d + 1) \to 0
\]
is equivalent to giving a line in $\Ext^1(\OO_V(d+1), F)$.
\end{lem}

\begin{proof}
Any extension
\[
0 \to F \to E \to \OO_V(d + 1) \to 0
\]
corresponds to an element in $\Ext^1(\OO_V(d+1), F)$. In Theorem \ref{thm:sequences_maximal_ch3} we have already shown that for semistable objects this is the Harder-Narasimhan filtration of $E$ once $E$ becomes tilt-unstable. The Harder-Narasimhan factors are unique. This means $E$ determines both $V$ and $F$ as a subobject of $E$. Scaling the map $F \to E$ does not change the isomorphism class of $E$. Moreover, if $E$ was a direct sum, it would not be stable. Therefore, we get a line in $\Ext^1(\OO_V(d+1), F)$.

Assume vice versa that we have a line in $\Ext^1(\OO_V(d+1), F)$. Choosing an arbitrary non-zero element on this line leads to a non-trivial extension
\[
0 \to F \to E \to \OO_V(d+1) \to 0.
\]
This object $E$ is strictly semistable along the induced wall $W$. By Theorem \ref{thm:sequences_maximal_ch3} $E$ satisfies the JH-property, and the only relevant destabilizing subobjects of $E$ above the wall could be either $F$ or $\OO_V(d+1)$. However, $F$ does not destabilize $E$ for purely numerical reasons, and the fact that the exact sequence does not split excludes $\OO_V(d+1)$.
\end{proof}

Next we have to construct the variety that we expect to be our moduli space. Let $\WW$ be the universal family on $M(2,-1,-\tfrac{1}{2},\tfrac{5}{6}) \cong \P^3$, and let $\VV \subset \Gr(3,4) \times \P^3 \cong \P^3 \times \P^3$ be the universal plane. We have projections
\begin{align*}
    p_{12} : \Gr(3,4) \times M(2,-1,-\tfrac{1}{2},\tfrac{5}{6}) \times \P^3 &\to \Gr(3,4) \times M(2,-1,-\tfrac{1}{2},\tfrac{5}{6}), \\
    p_{13} : \Gr(3,4) \times M(2,-1,-\tfrac{1}{2},\tfrac{5}{6}) \times \P^3 &\to \Gr(3,4) \times \P^3, \\
    p_{23} : \Gr(3,4) \times M(2,-1,-\tfrac{1}{2},\tfrac{5}{6}) \times \P^3 &\to M(2,-1,-\tfrac{1}{2},\tfrac{5}{6}) \times \P^3, \\
    p_3: \Gr(3,4) \times M(2,-1,-\tfrac{1}{2},\tfrac{5}{6}) \times \P^3 &\to \P^3.
\end{align*}
By Theorem \ref{thm:lange_base_change} and Lemma \ref{lem:c0_general_ext} we get that
\[
A := \lExt^1_{p_{12}}(p_{13}^* \OO_{\VV} \otimes p_3^* \OO(d + 1), p_{23}^* \WW) \cong Rp_{12*} \RHom(p_{13}^* \OO_{\VV} \otimes p_3^* \OO(d + 1), p_{23}^* \WW)[1]
\]
is a vector bundle. Let $\pi: \P(A^{\vee}) \to \Gr(3, 4) \times M(2,-1,-\tfrac{1}{2},\tfrac{5}{6})$ be the projection from the projective bundle of locally free rank one quotients of $A^{\vee}$ to its base. Furthermore, we have a relatively ample line bundle $\OO_{\pi}(1)$ on this projective bundle and a projection $q: \P(A^{\vee}) \times \P^3 \to \P(A^{\vee})$. By \cite[Corollary 4.5]{Lan83:families_extensions} there is an extension
\[
0 \to (\pi \times \id)^* p_{23}^* \WW \otimes q^* \OO_{\pi}(1) \to \UU \to (\pi \times \id)^* (p_{13}^* \OO_{\VV} \otimes p_3^* \OO(d + 1)) \to 0,
\]
such that the fibers of $\UU$ are in bijection with non-trivial extensions
\[
0 \to F \to E \to \OO_V(d+1) \to 0
\]
as previously. This family satisfies a universal property on the category of noetherian $(\P^3)^{\vee} \times M(2,-1,-\tfrac{1}{2},\tfrac{5}{6})$-schemes, but this is not the universal property we need on the category of noetherian $\C$-schemes. Regardless, the universal property of $M(2,0,-d, \tfrac{d^2}{2} + \tfrac{d}{2} + 1)$ implies that there is a bijective morphism $\P(A^{\vee}) \to M(2,0,d, \tfrac{d^2}{2} + \tfrac{d}{2} + 1)$. We are done if we can show that $M(2,0,d, \tfrac{d^2}{2} + \tfrac{d}{2} + 1)$ is smooth.

\begin{lem}
The moduli space $M(2,0,d, \tfrac{d^2}{2} + \tfrac{d}{2} + 1)$ is smooth.
\end{lem}

\begin{proof}
Applying the three functors $\RHom( \cdot, F)$, $\RHom( \cdot, \OO_V(d + 1))$, and $\RHom(E, \cdot)$ to
\[
0 \to F \to E \to \OO_V(d + 1) \to 0
\]
leads to
\[
\ext^1(E, E) = d(d-2) + 5 = \dim \P(A^{\vee}). \qedhere
\]
\end{proof}

\def\cprime{$'$} \def\cprime{$'$}

\end{document}